\documentclass[final,onefignum,onetabnum]{siamart171218}
\usepackage{lipsum}
\usepackage{amsfonts}
\usepackage{graphicx}
\usepackage{epstopdf}
\usepackage{algorithmic}
\newcommand{\algorithmicbreak}{\textbf{break}}
\newcommand{\BREAK}{\STATE \algorithmicbreak}
\usepackage{subfig}
\usepackage{multirow}
\usepackage{makecell}

\usepackage{glossaries}
\newglossary{symbols}{sym}{sbl}{List of Abbreviations and Symbols}
\makeglossary
\newglossaryentry{fn}{type=symbols,name={\ensuremath{F_n}},sort=fn,
description={Empirical (sample) distribution function}}
\newglossaryentry{fncon}{type=symbols,name={\ensuremath{F^{n^\ast}}},sort=fnc,description={$n$-fold convolution of the distribution function/distribution $F$}}

\ifpdf
  \DeclareGraphicsExtensions{.eps,.pdf,.png,.jpg}
\else
  \DeclareGraphicsExtensions{.eps}
\fi

\newsiamremark{remark}{Remark}
\newsiamremark{hypothesis}{Hypothesis}
\crefname{hypothesis}{Hypothesis}{Hypotheses}
\newsiamthm{claim}{Claim}

\headers{Online Learning and Matching for Resource Allocation Problems}{A. Boskovic, Q. Chen, D. Kufel, and Z. Zhou}

\title{Online Learning and Matching for Resource Allocation Problems\thanks{Submitted to the editors November 17, 2019. Completed under the guidance of Anna Ma, Department of Mathematics, University of California, Irvine (\email{anna.ma@uci.edu}) and Xinshang Wang, DAMO Academy, Alibaba US (\email{xinshang.w@alibaba-inc.com}).}}

\author{Andrea Boskovic\thanks{Department of Statistics, Amherst College, Amherst, MA 01002} (\email{aboskovic21@amherst.edu}). 
\and Qinyi Chen\thanks{Department of Mathematics, University of California, Los Angeles, Los Angeles, CA 90095} (\email{qinyichen@ucla.edu}). 
\and Dominik Kufel\thanks{Department of Physics, University College London, Gower St, Bloomsbury, London WC1E 6BT, United Kingdom} (\email{dominic.kufel@gmail.com}).
\and Zijie Zhou\thanks{Department of Mathematics, Purdue University, 610 Purdue Mall, West Lafayette, IN 47907} (\email{zhou759@purdue.edu}).
}

\usepackage{amsopn}

\makeatletter
\newcommand*{\addFileDependency}[1]{% argument=file name and extension
  \typeout{(#1)}% latexmk will find this if $recorder=0 (however, in that case, it will ignore #1 if it is a .aux or .pdf file etc and it exists! if it doesn't exist, it will appear in the list of dependents regardless)
  \@addtofilelist{#1}% if you want it to appear in \listfiles, not really necessary and latexmk doesn't use this
  \IfFileExists{#1}{}{\typeout{No file #1.}}% latexmk will find this message if #1 doesn't exist (yet)
}
\makeatother

\newcommand*{\myexternaldocument}[1]{%
    \externaldocument{#1}%
    \addFileDependency{#1.tex}%
    \addFileDependency{#1.aux}%
}

\ifpdf
\hypersetup{
  pdftitle={Online Learning and Matching for Resource Allocation Problems},
  pdfauthor={A. Boskovic, Q. Chen, D. Kufel, and Z. Zhou}
}
\fi

\myexternaldocument{ex_supplement}

\begin{document}

\maketitle

\begin{abstract}
In order for an e-commerce platform to maximize its revenue, it must recommend customers items they are most likely to purchase. However, the company often has business constraints on these items, such as the number of each item in stock. In this work, our goal is to recommend items to users as they arrive on a webpage sequentially, in an online manner, in order to maximize reward for a company, but also satisfy budget constraints. We first approach the simpler online problem in which the customers arrive as a stationary Poisson process, and present an integrated algorithm that performs online optimization and online learning together. We then make the model more complicated but more realistic, treating the arrival processes as non-stationary Poisson processes. To deal with heterogeneous customer arrivals, we propose a time segmentation algorithm that converts a non-stationary problem into a series of stationary problems. Experiments conducted on large-scale synthetic data demonstrate the effectiveness and efficiency of our proposed approaches on solving constrained resource allocation problems. 
\end{abstract}

\begin{keywords}
    online algorithms, resource allocation, traffic shaping, reinforcement learning, online convex optimization, non-stationary arrivals  
\end{keywords}

\begin{AMS}
90B05, 90B50, 90B60, 90C05 
\end{AMS}

\section{Introduction}
\label{sec:Introduction}
Resource allocation has been considered an important task by many e-commerce platforms, and it can essentially be formulated as a generalized online matching problem. In an electronic marketplace, products are placed for sale on a webpage as customers arrive sequentially, viewing the products and making purchase decisions. As each customer arrives, the platform needs to display corresponding items that they are likely to purchase. However, the tendency of each customer to purchase a certain product is unknown, and the revenue generated by the sale of different products varies. In a given session, we assume that customers arrive onto the webpage randomly over time. The task at hand is to find a way to match each customer to an item such that this matching maximizes the reward (i.e., the potential revenue generated through the sale of items) with respect to certain constraints, such as the stock of each item.

Resource allocation problems can be approached in either an offline or an online manner. The offline problem assumes that the sequence of customer arrivals is known in advance, while in the online problem, we consider customers arriving onto the webpage as following an unknown stochastic process. Oftentimes, customers arriving in an online manner are modeled as a stationary Poisson process, the rate of which is unknown beforehand. The offline algorithm optimizes multiple functions simultaneously \cite{agrawal2014dynamic}, whereas the online problem involves optimizing different sequences of functions at each time.  Although the offline problem is a less realistic problem, the optimal solution to the offline problem is necessary for the evaluation of regret in the online problem. Current work on the online problem \cite{agrawal2014bandits, agrawal2014fast, badanidiyuru2013bandits, cheung2018inventory} mainly focuses on its theoretical aspect and attempts to minimize the regret, which is a measure of how well the online algorithm is working in comparison to the offline algorithm and its optimal solution. Moreover, resource allocation problems are closely related to ad allocation, which is also studied in the context of online matching problems. Some notable examples include \textit{DisplayAds} \cite{feldman2009online} and \textit{AdWords} \cite{devanur2009adwords, mehta2013online}. 

In this paper, we propose several online algorithms for allocating products to users, which extend and improve previous work. We first approach the online stationary problem, in which the sequence of customer arrivals is unknown, but the arrival rates of customers are constant over time, by introducing an integrated algorithm that performs online learning and matching together. We then proceed to the non-stationary case, in which customer arrival rates vary over time, and propose another time segmentation algorithm that tackles the customer heterogeneity. We theoretically verify the convergence of average regret in our algorithms, and experimentally demonstrate their efficacy in providing near-optimal product recommendations.

The rest of the paper is organized as follows. \Cref{sec:LiteratureReview} discusses the background of the offline and online resource allocation problems, and some existing approaches that we rely upon. \Cref{sec:IntegratedAlgorithm} introduces an integrated algorithm that tackles the online problem with stationary customer arrivals. \Cref{sec:OnlineNonstationaryProblem} extends the problem to consider heterogeneous customer arrivals, and propose another algorithm that approximates a non-stationary problem into a series of stationary problems. \Cref{sec:Experiment} demonstrates experimentally the effectiveness of our proposed algorithms. \Cref{sec:Conclusion} and \Cref{sec:FutureWork} interpret the results of our work and propose future directions.

\section{Background}
\label{sec:LiteratureReview}
We first review some existing approaches for the offline and online matching problems that we build upon to design our online algorithms as well as outline the framework of each approach.

\subsection{Offline Problem}
\label{subsec:offline}
In the offline matching problem, we assume that the distribution of customer arrivals and the preference of customers are both known. Therefore, we can simply optimize the potential revenue by solving the following linear program. Here, $j$ indexes the customers, where the total number of customers is $m$, and $i$ indexes the items, where the total number of items is $n$. Further, $r_i$ refers to the reward, or revenue, for the company when a particular customer purchases item $i$, $P_{ij}$ is the customer preference matrix, which contains the probability of customer $j$ purchasing item $i$ given they were offered item $i$, and $\Bar{P_j}$ is the maximum value of $P_{ij}$ for each customer, i.e., $\Bar{P_j} = \max_i P_{ij}$. Additionally, $x_{ij}$ refers to the probability that customer $j$ is recommended item $i$. A summary of notation can be seen in \Cref{app:Notations}.
\begin{equation} 
\label{eqn:lp} 
\max_{\substack{x_{ij} \\ i \in [n] \\ j \in [m]}}\sum_{i \in [n]}\sum_{j \in [m]} r_iP_{ij}x_{ij} - \mu\sum_{j \in [m]}\Bar{P_j}\sum_{i \in [n]}x_{ij}\log{x_{ij}},
\end{equation}

\begin{equation*}
\begin{aligned}
\mathrm{s.t.} & \sum_{j \in [m]}P_{ij}x_{ij} \leq b_{i}, \forall i \in [n]; \\
& \sum_{i \in [n]}x_{ij} = 1, \forall j \in [m]; \\
& x_{ij} \geq 0, \forall i \in [n], \forall j \in [m].
\end{aligned}
\end{equation*}
The maximization problem in the primal form can also be solved as a minimization problem in the dual form in the following way: 
\begin{equation} \label{eqn:dual}
f(\Lambda):=\mu \sum_{j \in [m]}\Bar{P_j} \log Z_j + \langle \Lambda,b \rangle,
\end{equation}
where $Z_j = \sum_{i \in [n]}\exp{(\frac{(r_i-\Lambda_i)P_{ij}}{\Bar{P_j}\mu})}$ is a normalization factor, and $\mu$ accounts for regularization \cite{zhong2015stock}, ensuring that our linear program is strongly convex and therefore has only one optimal solution. Given optimal $\Lambda$, the solution for $x$ is then:
\begin{equation}\label{eqn:primal}
x_{ij} = \frac{1}{Z_j}\exp{\frac{(r_i-\Lambda_i)P_{ij}}{\Bar{P_j}\mu}}.
\end{equation}

The primal formulation \eqref{eqn:lp} can be converted to its dual form \eqref{eqn:dual} by means of Lagrangian duality. This is a well-studied topic in optimization, and more details can be found in \cite{hazan2016intro}. To obtain the optimal solution to the offline matching problem, various first-order optimization algorithms can be applied, such as gradient descent (GD) and stochastic gradient descent (SGD). In this work, the objective function is minimized via GD with lingering radius ($\mathrm{GD}^{\mathrm{lin}}$), a less computationally expensive, state-of-the-art method \cite{allen2018lingering} well-suited for solving resource allocation problems.

\subsection{Online Stationary Problem}

The goal of an online matching algorithm is to recommend products to customers as they arrive sequentially onto the webpage in a way that not only maximizes reward, but also satisfies budget constraints. One difficulty of the online problem is that as each customer arrives, their preference for any particular item is unknown and must be learned in real time. To obtain a prediction of the customer preferences in advance, e-commerce platforms often divide the customers into different types, according to their demographics or other information. In the most simplified online problems, each type of customer is assumed to arrive as a stationary Poisson process. To learn their preferences $P_{ij}$, which correspond to the likelihood that a customer from type $j$ buys item $i$, we apply reinforcement learning techniques. By utilizing knowledge about the purchases of previous customers, we make product allocation decisions for future arriving customers. 

One commonly used technique to take the best possible action to maximize reward, or to determine the best product to recommend to each customer type, is the Upper Confidence Bound (UCB) algorithm \cite{Auer2003UCB}. The UCB algorithm is considered ideal for our purposes mainly because it is not greedy, i.e., it does not always recommend an item to a specific customer type if that item maximizes reward at a particular time. The algorithm exemplifies the principle of optimism in face of uncertainty, recommending items to each customer type until it exceeds some upper bound of certainty of the expected reward of that item's recommendation. This property allows us to obtain an accurate estimate of $P_{ij}$ fairly early on, thus enabling us to achieve a more accurate solution to the optimization problem. This approach is discussed in more detail in \Cref{sec:IntegratedAlgorithm}.

In addition to using the UCB algorithm to recommend products to users, we use online gradient descent, an online convex optimization method, to compute the gradient of the objective function of each arriving customer, which is then used to update the value of our dual variable $\Lambda$ \cite{zinkevich2003online}. The online convex optimization component of the algorithm is crucial in measuring the performance of the online integrated algorithm. Specifically, we seek to minimize regret, which is defined as follows: 
\begin{equation}
\label{eqn:regret}
\begin{split}
    \min_{\Lambda_{t}, t\in [T]} \mathrm{regret}_T = \sum_{t=1}^T \min(f_t(\Lambda_t, P^{(t)})) - \sum_{t=1}^T f_t(\Lambda^\ast, P^\ast).
\end{split}
\end{equation}
The regret function essentially compares the online problem for each arriving customer $t$ to the optimal solution to the offline problem, where we know the sequence of functions $\{f_1,f_2,...,f_T\}$ in advance \cite{shalev2012online}. In other words, regret acts as a metric that uses the offline problem as a benchmark for the online problem. The goal in solving the online problem is to minimize this regret function, thus minimizing the loss incurred due to error in optimization. 

\subsection{Online Non-stationary Problem}

In the online non-stationary problem, we consider a more realistic case: different types of customer arrive as non-stationary Poisson processes, in which their arrival rates are functions of time. As in the integrated algorithm, we consider the regret of the online non-stationary algorithm, defined in \Cref{eqn:regret}, and we again aim to minimize this regret function. Although minimal literature exists on problems with non-stationary stochastic customer arrivals, \cite{Stein2018AdvanceSR} discusses a non-stationary stochastic demand problem.

\section{Online Integrated Algorithm}
\label{sec:IntegratedAlgorithm}

We now consider customer arrivals onto a webpage in an online, or sequential, manner. Additionally, we assume no previous knowledge of customer preferences $P_{ij}$, and learn this value as customers arrive. The customers are assumed to arrive following a stationary Poisson process, where the Poisson arrival rates are known. In this section, we describe the formulation of the online stationary problem, and introduce an online integrated algorithm that combines the Upper Confidence Bound (UCB) algorithm, which learns customer preferences, with Online Gradient Descent (online GD), which tackles the optimization component of the problem. By performing online learning and optimization together, the integrated algorithm thus allows us to recommend the optimal product to each customer, and study their purchasing behaviors at the same time.

\subsection{Mathematical Formulation}
In the online stationary problem, we assume that a total of $T$ customers arrive over the entire time period. The customers arrive in a sequential manner, and when the $t^{th}$ customer arrives, the only information we have is the information about the previous customers. Our objective is to maximize the total expected reward for all customers by maximizing reward for any given $t^{th}$ customer, where $t \in [T]$. Therefore, for a particular customer $t$, we wish to solve the following maximization problem:
\begin{equation}
\begin{split}
\label{eqn:online_lp_1}
    \max_{x_{it}} 
    \sum_{i=1}^{n}  r_i P_{it} x_{it},
\end{split}
\end{equation}

\begin{equation*}
\begin{aligned}
\mathrm{s.t. } & \sum_{i=1}^{n} x_{ij} = 1 \ \ \forall{j \in [t]}, \ x_{ij} \geq 0; \\
& \sum_{j=1}^{t}P_{ij} x_{ij} \leq b_i \ \ \forall{i \in [n]}.
\end{aligned}
\end{equation*}

The last constraint here is based not only on the current customer but also on all the customers that have previously arrived. For simplicity, we have left out the regularization term in this formulation.

Problem \eqref{eqn:online_lp_1} above can be converted into the following dual problem:
\begin{equation}
\begin{split}
    \min_{\Lambda}f_t(\Lambda, P) = \min_{\Lambda} \left(\mu \overline{P_t} \log{Z_t} + \langle \Lambda,b \rangle - \sum_{i=1}^{n} \Lambda_i \sum_{j=1}^{t-1} P_{ij} y_{ij}\right).
\end{split}
\label{eq:solvefun}
\end{equation}
In the online setting, we seek to minimize \eqref{eq:solvefun}. In order to evaluate the performance of our online algorithm, we first define the regret function \cite{zinkevich2003online}, which is obtained comparing our online dual objective against the optimal dual objective obtained in the corresponding offline problem:
\begin{definition}\label{reg}
Given an online algorithm and online minimization problem \eqref{eq:solvefun}, the regret of the algorithm at time $T$ is:
$$\mathrm{regret}_T = \sum_{t=1}^T f_t(\Lambda_t, P^{(t)}) - \sum_{t=1}^T f_t(\Lambda^\ast, P^\ast),$$
where $\Lambda^\ast$ denotes the optimal dual variable in the offline problem and $P^\ast$ is the underlying ground truth customer preference matrix. At each iteration, we obtain the preference matrix $P^{(t)}$ and dual variable $\Lambda_t$. Note that here $\sum_{t=1}^T f_t(\Lambda^\ast, P^\ast)$ is simply the optimal offline dual for the $t^{\mathrm{th}}$ customer, which matches \eqref{eqn:dual}. Additionally, we define the average regret to be $\frac{\mathrm{regret}_T}{T}$.
\end{definition}
Our goal of solving the online problem is to minimize this regret function, thus minimizing the loss incurred due to error in optimization. Note that
\begin{equation*} 
\begin{split}
    \min_{\Lambda_{t}, t\in [T]} \mathrm{regret}_T = \sum_{t=1}^T \min(f_t(\Lambda_t, P^{(t)})) - \sum_{t=1}^T f_t(\Lambda^\ast, P^\ast) ,
\end{split}
\end{equation*}
i.e., minimizing the regret does not change the second term because the values in the summation are fixed. 

As mentioned, it is oftentimes too computationally expensive to learn the purchasing behavior of every single customer and minimize the dual variable for each of them. We therefore group the customers into different types based on their demographics---as e-commerce platforms tend to do in practice---since customers from the same background tend to display similar shopping behaviors. We assume that there are $m$ types of customers, and the customer preferences in each type are i.i.d. We let the preference matrix $P_{ij}$  represent the probability that any customer of type $j$ buys item $i$, instead of the preference of a single customer. Additionally, we assume the customers of type $j$ arrive as a stationary Poisson process of rate $\lambda_j$. Therefore by the superposition property of Poisson processes, we know that the probability that the customer arrival is of type $j$ is $\frac{\lambda_j}{\sum_{s=1}^m \lambda_s}$.

To reflect the changes in our model, we also make modifications to \eqref{eqn:online_lp_1}. The primal objective for the $t^{th}$ customer is now:
\begin{equation}
\label{eqn:online_lp_2}
    \max_{x_{ij}} \sum_{i=1}^n \sum_{j=1}^m r_iP_{ij}x_{ij}\frac{\lambda_j}{\sum_{s=1}^m \lambda_s},
\end{equation}

\begin{equation*}
\begin{aligned}
   \mathrm{s.t. } & \sum_{j=1}^m \frac{\lambda_j}{\sum_{s=1}^m \lambda_s}P_{ij}x_{ij} \leq \frac{b_i}{T}, \ \ \forall{i \in [n]}; \\
   & \sum_{i=1}^n x_{ij}=1, \ \  \forall{j \in [m]}; \\
   & x_{ij} \geq 0, \ \ \forall{i \in [n], \ j \in [m]}.
\end{aligned}
\end{equation*}
Note that \eqref{eqn:online_lp_2} now reflects the expected revenue we would obtain from the $t^{th}$ customer arrival. As before, we can convert it to the following dual problem: 
\begin{equation}
\label{eqn:online_lp_3}
    \min_{\Lambda}f_t(\Lambda, P)= \min_{\Lambda} \left(\mu \sum_{j=1}^m\frac{\lambda_j}{\sum_{s=1}^m \lambda_s} \overline{P_j} \log(Z_j) + \frac{1}{T} \langle \Lambda,b \rangle \right),
\end{equation}
where $Z_j=\sum_{i=1}^n \exp(\frac{(r_i-\Lambda_i)P_{ij}}{\mu \overline{P_j}})$. Note that while $f_t$ denotes the objective function related to the $t^{th}$ customer, $f_t$ does not depend on $t$. We can again obtain $x_{ij}$ by applying \eqref{eqn:primal}.

Another issue that one needs to take note is that customer preference $P$ is initially unknown. Now, not only do we need to solve the online stationary problem, we also need to gradually learn $P$ and to keep updating it as customers arrive. Therefore, when solving the minimization problem in \eqref{eqn:online_lp_3}, the variable $P_{ij}$ will change as customers continue arriving. In the following section, we describe an integrated algorithm that allows us to learn $P$ and solve the optimization problem simultaneously.

\subsection{Upper Confidence Bound (UCB) Algorithm}
In \Cref{alg:ucb}, we introduce the UCB algorithm \cite{lattimore2018bandit} used as part of our integrated algorithm. Here, we let $D \in \{0,1\}^{n \times T}$ denote a binary reward matrix, where each entry $D_{it}$ denotes whether or not the $t^{\mathrm{th}}$ customer buys item $i$. By the time of the $t^{\mathrm{th}}$ customer arrival, we let $N_i(t)$ denote the number of times item $i$ has been selected and $R_i(t)$ be the amount of rewards we have already collected by assigning item $i$. The average reward is denoted as $\overline{r}_i(t)=R_i(t)/N_i(t)$. 
We define our UCB function as follows:
\begin{equation*}
\mathrm{UCB}_i(t-1) =
\begin{cases} 
      \infty & N_i(t-1) = 0 \\
      \overline{r}_i(t-1) + \sqrt{\frac{3\log(t)}{2N_i(t-1)}} & \mathrm{otherwise.}
   \end{cases}
\end{equation*}
Note that the definition of the upper confidence bound can in fact be changed depending on how much importance we place on the exploration component.

\begin{algorithm}
\caption{Upper Confidence Bound (UCB) Algorithm}
\hspace*{\algorithmicindent} \textbf{Input}: number of customer arrivals $T$, reward matrix D \\
\hspace*{\algorithmicindent} \textbf{Output}: item assignments $\{i^{(t)}\}_{t = 1, ..., T}$
\begin{algorithmic}
\FOR{t = 1, ..., T}
    \STATE Choose the item to assign: $i^{(t)} = \mathrm{argmax}_i{\mathrm{UCB}_i(t-1)}$ \\
    \STATE Observe reward $D[i^{(t)}, t]$ \\
    \STATE $N_i(t) = N_i(t-1)+1; R_i(t) = R_i(t-1) + D[i^{(t)}, t]$
    \\
\ENDFOR
\end{algorithmic}
\label{alg:ucb}
\end{algorithm}

\subsection{Online Gradient Descent (Online GD)}
Online GD \cite{hazan2016intro} is an algorithm similar to offline gradient descent. In the offline problem, since all the data is known at the start of the matching process, we can compute the gradient of the full objective function. However, in the online problem, since customers arrive one by one, our data set grows over time as we learn more about the item preferences of each customer type. Therefore, we can only use the data we have at a particular time to compute gradients. Thus, we only iterate through the data set once, unlike in the offline GD, where we loop through the data many times.

When applying online GD, we start from an initial $\Lambda_0 \in \kappa$, where $\kappa$ is a convex set. Then we iterate through $t = 1, ..., T$, and at each iteration, we update $\Lambda_t$ in the following way:
$$y_{t+1} = \Lambda_t-\eta_t \nabla f_t(\Lambda_t),$$
$$\Lambda_{t+1} = \mathrm{proj}_{\kappa} (y_{t+1}).$$
Here, $\eta_t \in \mathbb{R}$ is the step size, and $\mathrm{proj}_{\kappa} (y_{t+1})$ is the projection of $y_{t+1}$ onto a convex set $\kappa$. If $y_{t+1} \in \kappa$ then $\Lambda_{t+1}=y_{t+1}$. If $y_{t+1} \not\in \kappa$, then $\Lambda_{t+1}= \mathrm{proj}_{\kappa}(y_{t+1}) = \mathrm{argmin}_{s\in \kappa} {\Vert y_{t+1}-s \Vert}$.

\subsection{Integrated Algorithm}
Combining the UCB algorithm and online GD, we create an efficient integrated algorithm that solves the online stationary problem. Our integrated algorithm relies on a learning component that updates customer preference $P_{ij}$, as well as an optimization component that finds the optimal assignment of items that results in the highest expected reward. We assume that when the $t^{th}$ customer arrives, we will first observe their type $j$ and assign them to an item $i$ using the UCB algorithm. Then, based on whether the customer of type $j$ purchases the item $i$ or not, we update $P^{(t)}$ to reflect a more accurate customer preference matrix: each entry $P^{(t)}_{ij} = R^{(t)}_{ij}/N^{(t)}_{ij}$, where $R^{(t)}_{ij}$ is the total number of times customers of type $j$ purchase item $i$, and $N^{(t)}_{ij}$ is the total number of times item $i$ gets assigned to customers of type $j$ until time $t$. We then use this $P^{(t)}$ in \eqref{eqn:online_lp_3} and apply online GD to get the solution for the dual variable. If the $P^{(t)}$ that we get at each iteration converges, we can halt the UCB algorithm and only run online GD until the dual variables $\Lambda_t$ also converge. 

Note that since $P^{(t)}$ does not necessarily reflect the true preference matrix, the optimization problem that we solve changes at each iteration as we update $P^{(t)}$. In the following section, we theoretically show that as long as the number of customer arrivals are sufficient, $P^{(t)}$ eventually converges to the true $P_{ij}$. Our analysis leads to the algorithm in \Cref{alg:integratedalg}.

\begin{algorithm}
\caption{Online Integrated Algorithm}
\label{alg:integratedalg}

\hspace*{\algorithmicindent} \textbf{Input}: Customer arrivals $t = 1, ..., T$, number of customer types $m$, number of items $n$, budgets $b \in \mathbb{R}^n$,  rewards $r \in \mathbb{R}^n$, initial preference matrix $P^{(0)} \in \mathbb{R}^{m\times n}$, initial dual variable $\Lambda_0 \in \mathbb{R}^n$, maximum number rounds of UCB $R_{\mathrm{max}}$.\\
\hspace*{\algorithmicindent} \textbf{Output}: item assignments for $t = 1, ..., T$. \\
\begin{algorithmic}
\vspace{-3mm}
\FOR{t = 1, ..., T}
    \STATE Observe the type of this customer: $j = 1, ... ,m$. \\
    \IF {$\left\lVert P^{(t)}-P^{(t-1)}\right\rVert > \epsilon$ \AND $t\leq R_{\mathrm{max}}$}
        \STATE Assign item $i$ with maximum UCB value for customer type $j$. \\
    \ELSE
        \STATE Assign item $i$ with $\Lambda_\mathrm{t-1}.$\\
    \ENDIF
    \STATE Update $P^{(t)}_{ij} = R^{(t)}_{ij}/N^{(t)}_{ij}$.\\
    \STATE Define $f_t(\Lambda) = f_t(\Lambda, P^{(t)})$ according to \eqref{eqn:online_lp_3}.\\
    \STATE $\Lambda_{t} = \mathrm{proj}_{\kappa}\{\Lambda_{t-1}-\eta_{t-1} \nabla f_t(\Lambda_{t-1})\}.$\\
\ENDFOR 
\end{algorithmic}
\end{algorithm}

\subsection{Upper Bound of Average Regret}
Recall that in the online model, our goal is to minimize the regret, as defined in \Cref{reg}. Here, we show that using the integrated algorithm, the average regret converges to zero as the number of customer arrivals approaches infinity.
\begin{theorem}
\label{theorem:stationary_regret}
Consider \Cref{alg:integratedalg}, we have: 
$$\limsup\limits_{T \rightarrow \infty} \frac{\mathrm{regret}_T}{T}=0.$$
\end{theorem}
\begin{proof}
Recall that the optimal offline solution is the minimizer of \eqref{eqn:dual}, which can be re-written as $\Lambda^*=\mathrm{argmin}_{\Lambda\in \kappa}{\sum_{t=1}^{T} f_t(\Lambda, P^\ast)}.$ 
 Assuming $\Lambda^* \in \kappa$, when projecting $y_{t+1}$ onto $\kappa$, we must have 
 $\Vert \Lambda_{t+1}-\Lambda^*\Vert 
 = \Vert \mathrm{proj}_\kappa(y_{t+1}) - \Lambda^* \Vert 
 \leq \Vert y_{t+1}-\Lambda^*\Vert$, where 
 $y_{t+1} = \Lambda_{t}-\eta_{t} \nabla f_{t+1}(\Lambda_{t})$. 
 Define $\bigtriangledown_t=\bigtriangledown_{\Lambda} f_t(\Lambda_t, P^{(t)})$. We have that 
\begin{equation*}
    \Vert y_{t+1}-\Lambda^*\Vert^2 = \Vert \Lambda_t- \eta_t\bigtriangledown_t -\Lambda^*\Vert^2=\Vert \Lambda_t-\Lambda^*\Vert^2+\eta_t^2 \Vert \bigtriangledown_t \Vert^2 -2\eta_t \langle \bigtriangledown_t,\Lambda_t-\Lambda^* \rangle.
\end{equation*}
By the convexity of $f_t$, we can get 
\begin{equation} 
\label{eqn:convexity}
\begin{aligned}
    f_t(\Lambda_t, P^\ast)-f_t(\Lambda^*, P^\ast) & \leq \langle \bigtriangledown_t,\Lambda_t-\Lambda^* \rangle \\ & 
    \leq \frac{1}{2\eta_t}(\Vert \Lambda_t-\Lambda^* \Vert^2-\Vert \Lambda_{t+1}-\Lambda^* \Vert^2)+\frac{\eta_t}{2}\Vert \bigtriangledown_t\Vert^2.
\end{aligned}
\end{equation}
Now, we fix $\eta_t$ and let $D$ be the diameter of the convex set $\kappa$. Let $G$ be s.t. $\Vert \bigtriangledown_t\Vert \leq G$ for all $1 \leq t \leq T$ and for all $\Lambda \in \kappa$. We define $\eta = \eta_t=\frac{D}{G\sqrt{T}}$ for $1 \leq t \leq T$. If we sum over $t$ for \eqref{eqn:convexity}, we get
\begin{equation*}
\begin{aligned}
    \sum_{t=1}^{T}(f_t(\Lambda_t, P^\ast)-f_t(\Lambda^*, P^\ast)) 
    & \leq \frac{1}{2\eta}\Vert \Lambda_0-\Lambda^* \Vert^2+\frac{\eta}{2}\sum_{t=1}^{T}\Vert \bigtriangledown_t\Vert^2
    \\ & \leq \frac{1}{2\eta} D^2+\frac{\eta}{2} T G^2
    \\ & = GD\sqrt{T}.
\end{aligned}
\end{equation*}
Note that the function $f_t$ is different for the online and offline problems because in the online problem, the preference $P$ gets updated at each iteration. We have:
\begin{equation} 
\label{eqn:online_eqn_1}
    f_t(\Lambda_t, P^{(t)})-f_t(\Lambda^*, P^\ast) \leq \vert f_t (\Lambda_t, P^{(t)})-f_t(\Lambda_t,P^\ast) \vert  + \vert f_t(\Lambda_t, P^\ast)-f_t(\Lambda^\ast, P^\ast) \vert .
\end{equation}
If we sum over \eqref{eqn:online_eqn_1} for $1 \leq t \leq T$, we get the following:
\begin{equation*}
    \mathrm{regret}_T \leq \sum_{t=1}^{T} \vert f_t (\Lambda_t, P^{(t)})-f_t(\Lambda_t,P^\ast) \vert + GD\sqrt{T}.
\end{equation*}
Here, $\vert f_t (\Lambda_t, P^{(t)})-f_t(\Lambda_t,P^\ast) \vert$ represents the regret resulting from approximating $P^\ast$ with $P^{(t)}$. If we apply the UCB algorithm to obtain the approximations $P^{(t)}$ for $t = 1, ..., T$, the total regret after $T$ iterations is $O(T \log T)$ \cite{lattimore2018bandit}. That is, there exists $C > 0$ such that
$$\sum_{t=1}^T |f_t(\Lambda_t,P^{(t)})-f_t(\Lambda_t,P^{\ast})| \leq C\sqrt{T\log T}.$$
Thus, we have
\begin{equation*}
   \limsup\limits_{T \rightarrow \infty} \frac{\mathrm{regret}_T}{T} \leq \limsup\limits_{T \rightarrow \infty} C \sqrt{\frac{\log T}{T}}+GD \sqrt{\frac{1}{T}} = 0，
\end{equation*}
where the last equality follows from L'Hopital's Rule. Thus, the average regret converges to 0 when $T \rightarrow \infty$.
\end{proof}

\section{Online Non-Stationary Problem} 
\label{sec:OnlineNonstationaryProblem}

Realistically, customers do not always arrive following a stationary Poisson process. We now extend our previous discussion to consider a more practical setting in which the customers arrive following non-stationary Poisson processes. Removing the assumption of the stationary Poisson processes leads to a more complex formulation of the online linear program that cannot simply be solved using \Cref{alg:integratedalg}. In existing literature of the online optimization problem with heterogeneous customer arrivals, if the arrival rates and customer preference $P_{ij}$ are both known, we can apply methods such as the Large-or-Small Algorithm \cite{Stein2018AdvanceSR}. However, there is no existing algorithm that can perform online optimization without previous knowledge of arrival rates or customer preference $P_{ij}$. The difficulty lies in that if we perform online learning for $P_{ij}$, the dual variables do not converge. Moreover, if the arrival processes are modeled as non-stationary Poisson processes, then the probability that the next customer arrival comes from type $j$ is almost impossible to calculate. 

In this section, we propose a time segmentation algorithm that can approximate this probability and convert the non-stationary arrival problem into a series of stationary problems. We can then solve each stationary problem using the method discussed in \Cref{sec:IntegratedAlgorithm}. We give a detailed description of this algorithm in \Cref{subsec:non_stationary_algo}, and provide an upper bound of the average regret of this approach in \Cref{subsec:upper_bound_avg_regret}.

\subsection{Algorithm Description}
\label{subsec:non_stationary_algo}

In the online stationary problem, we obtain the probability that the next customer arrival is of type $j$ by directly invoking the superposition property of Poisson processes. However, in a non-stationary Poisson process, this probability continuously depends on time and thus cannot be simply computed as a constant. 
To tackle this difficulty that arises, we assume that the Poisson rate function $\lambda_j(t)$ changes slowly inside a sufficiently small time interval. This is a realistic assumption since within a short time period---for example, 10 minutes---it is unlikely that the density of customer arrivals would change drastically. The following discussion thus relies on the assumption that the amount of the change of arrival rate function $\lambda_j(t)$ in a specific time segment $I$, which is defined by $\max_{t \in I} \lambda_j(t) - \min_{t \in I} \lambda_j(t)$, is bounded by some constant. As we shall see later, to perform an accurate and computationally feasible approximation, we would need both this constant to be sufficiently small, and the length of the time interval $I$ to be reasonably large.
We make a further assumption that the rate function $\lambda_j(t)$ is bounded in any given time interval for all $ 1\leq j\leq m$. Note that the rate functions may still be discontinuous. An example rate function that satisfies the above assumptions would be: 
\begin{equation*}
\lambda(t) =
\begin{cases} 
      0.5\sin(t) + 30 & 0 \leq t \leq 1 \\
      0.01t + 5 & 1 \leq t \leq 2
   \end{cases}.
\end{equation*}
This represents a realistic setting when a website experiences heavier traffic during the first hour, while the customer arrivals slow down in the second hour; however, arrival rates within one hour do not change drastically. 

We proceed to describe the main ideas behind the algorithm for the online non-stationary problem. Recall the Piecewise Constant Approximation Theorem:
\begin{theorem} \label{theorem:piece}
   If $f$ is a continuous function defined on compact $D \in \mathbb{R}$, it can be uniformly approximated by a piecewise constant function.
\end{theorem}
\Cref{theorem:piece} implies that we can use a piecewise constant function to approximate each arrival rate function. In other words, there is a series of time segments $\{I\}$ in which
$$\max_{t \in I} \lambda_j(t) - \min_{t \in I} \lambda_j(t) \leq \epsilon,$$ 
for a sufficiently small number $\epsilon$.
Inside each time segment, the rate functions can be approximated as constants. In doing so, the non-stationary processes can be approximated with multiple stationary Poisson processes in small time segments. However, it is not computationally feasible to entirely rely on this approach. If we divide the time span into small time segments based on the approximation in \Cref{theorem:piece}, the length of time segments has to be extremely small in some cases in order to achieve the desired accuracy. In such cases, we end up dealing with too many time segments. Solving online stationary problems in a large number of time segments in these instances leads to excessive computational cost. Therefore, we can only apply piecewise constant approximation when the rate functions $\lambda_j(t)$'s all change extremely slowly in a time segment of reasonable length. In our algorithm, we would refer to these time segments as type A.

If the rate functions change moderately slowly, and we are unable to find a time segment of sufficient length on which to perform piecewise constant approximation, we turn to a different approach. Instead of approximating the arrival rates as constants, we instead approximate the probability that the next customer arrival is of type $j$ directly. In particular, we find time segments in which the difference between the upper and lower bounds of this probability is small. We then pick a random value between the upper and lower bounds to be the approximated probability in that time segment, without incurring significant loss in accuracy. In our algorithm, we refer to these time segments as type B. To identify such a time segment, we first introduce \Cref{divide}, which can be proved by contradiction.
\begin{lemma} 
\label{divide}
If $f$ is a continuous function defined in some domain $D$, one can divide $D$ into disjoint segments $I_1, I_2, ...,$ s.t. in each segment, the amount that $f$ changes is bounded by a given threshold $v$.
\end{lemma}
By choosing an appropriate threshold $v \gg \epsilon$ that bounds the amount of the change of rate functions, one can control the amount of inaccuracy incurred by the approximation of the probability of next customer arrival being of type $j$. This is captured in the following Theorem:
\begin{theorem}
    Assume that during time period $[t_1, t_2]$, for some $v \gg \epsilon > 0$, $$\max_{t \in [t_1, t_2]} \lambda_j(t) - \min_{t \in [t_1, t_2]}\lambda_j(t)\leq v \hspace{5mm} \forall 1 \leq j \leq m.$$ Let $U(j)$, $L(j)$ denote the upper and lower bound of probability that the arriving customer is of type $j$ during the time period $[t_1, t_2]$. Let $Y = \sum_{j=1}^m \max_{s \in [t_1, t_2]} \lambda_j(s)$ and $y = \sum_{j=1}^m \min_{s \in [t_1, t_2]} \lambda_j(s)$.
    Then,
    \begin{equation}\label{eq:Uj}
        U(j) \leq \frac{\max_{s \in [t_1, t_2]}\lambda_j(s)}{y},
    \end{equation}
    \begin{equation}\label{eq:Lj}
        L(j) \geq \frac{\min_{s \in [t_1, t_2]}\lambda_j(s)}{Y}.
    \end{equation}
    It follows that if $\delta(j)=U(j)-L(j)$, we have:
    \begin{equation}
    \label{eq:bound}
        \delta(j) \leq \frac{mv^2+(y+m\min_{s \in [t_1, t_2]}\lambda_j(s))v}{y^2+mvy}.
    \end{equation}
\end{theorem}
Equation \eqref{eq:bound} demonstrates the relationship between the chosen threshold $v$ and the difference between the upper and lower bound $\delta(j)$. In the case when the rate functions change moderately slowly, one can obtain an approximation of the probability that the next customer arrival is of type $j$ by controlling the changes of the arrival rate functions inside each time segment. It is noteworthy that $\delta(j)$ not only depends on $v$, but also depends on rate functions inside specific time segments. Therefore, even if we seek a constant confidence bound $\delta(j)$, the desired values of $v$ will vary at different times.

\begin{algorithm}
\caption{Time Segmentation}
\hspace*{\algorithmicindent} \textbf{Input}: rate functions $\lambda_j(t)$, time span $[t_0, t_\mathrm{end}]$, parameters $\epsilon, \delta, d > 0$ \\
\hspace*{\algorithmicindent} \textbf{Output}: time segments $I_1, I_2, ..., I_l$ 
\begin{algorithmic}
\WHILE{$t < t_\mathrm{end}$}
    \FOR {j = 1, ..., m}
        \STATE compute the largest $t_j$ s.t. $\max_{s \in [t,t_j]}\lambda_j(s) - \min_{s \in [t,t_j]}\lambda_j(s)\leq \epsilon$
    \ENDFOR
    \STATE $t^* \gets \min_{j \in [m]}t_j$
    \IF{$t^*-t \geq d$}
        \STATE Add $[t, t^*]$ as one of the time segments, and mark it as type A; $t \gets t^*$
        \BREAK
    \ELSE
        \STATE Solve $mv^2+(\sum_{j}\lambda_j(t) + m\lambda_j(t) - \delta m\sum_{j}\lambda_j(t))v-\delta (\sum_{j}\lambda_j(t))^2=0$ for $v$
        \FOR{j = 1, ..., m}
        \STATE compute the largest $t'_j$ s.t. $\max_{s \in [t,t'_j]}\lambda_j(s) - \min_{s \in [t,t'_j]}\lambda_j(s)\leq v$
        \ENDFOR
        \STATE $t^* \gets \min_{j \in [m]}t'_j$
        \STATE Add $[t, t^*]$ as one of the time segments, and mark it as type B; $t \gets t^*$
    \ENDIF
\ENDWHILE
\end{algorithmic}
\label{alg:time} 
\end{algorithm}

In \Cref{alg:time}, we propose a time segmentation algorithm that divides the entire time span into small time segments. Note that we always first look for time segments of type A, in which we can perform piecewise constant approximation. However, if such time segments do not have sufficient length, we turn to look for time segments of type B, in which we can approximate the probability of the next customer arrival being a particular type. For the time segments \{$I^{(A)}$\} that are marked as type A, we select a random $t \in I^{(A)}$ and approximate the arrival rates as a fixed number: $\lambda_j(t) \approx \lambda_j$. We can then solve the non-stationary problem in this time segment as an online stationary problem. On the other hand, for the time segments \{$I^{(B)}$\} that are marked as type B, we compute $U(j)$ and $L(j)$ as in \eqref{eq:Uj} and \eqref{eq:Lj}. We then approximate $\frac{\lambda_j}{\sum_{s=1}^{m}\lambda_s}$ as a random number between $U(j)$ and $L(j)$. In this way, we again approximate the non-stationary problem as a stationary problem in this time segment. To solve the online stationary problem in each time segment, we simply apply the integrated \Cref{alg:integratedalg}, using both UCB and online GD. In practice, when the number of customer arrivals in each time segment is large, the dual variables should converge before reaching the end of the time segment.

\subsection{Upper Bound of Average Regret}
\label{subsec:upper_bound_avg_regret}
When solving online non-stationary problem, we first apply \Cref{alg:time} to divide the time span into small time segments, and then apply \Cref{alg:integratedalg} to solve the online stationary problem within each time segment. As in \Cref{sec:IntegratedAlgorithm}, we now provide an analysis of the average regret bound of this approach. Since the online non-stationary algorithm involves both the time segmentation algorithm and the integrated algorithm, the regret computation here would require our previous analysis of UCB and online GD. Our analysis over the average regret bound here will be focused on a single time segment. 
\begin{theorem}
\label{theorem:non_stationary_regret}
Consider the online non-stationary algorithm described in \Cref{subsec:non_stationary_algo}, in a specific time segment $I$, we have:
\begin{equation}
\label{eqn:non_stationary_avg_regret}
    \limsup\limits_{T \rightarrow \infty} \frac{\mathrm{regret}_T}{T} \leq R\delta
\end{equation}
where $T$ is total number of customer arrivals in $I$, $\delta$ is the confidence bound used for type B time segments and $R = m(\mu\log n + r^*)$, where $r^*=\max_{i \in [n]} r_i$.
\end{theorem}
\begin{proof}
As in the proof of \Cref{theorem:stationary_regret}, we define
$f_t(\Lambda, P)$ to be the dual objective at time $t$. We let $\Lambda^\ast=\mathrm{argmin}{\sum_{t=1}^{T} f_t(\Lambda, P^\ast)}$, where $P^\ast$ is the ground truth preference matrix. We denote $\Lambda_t$ and $P^{(t)}$ as the dual variable and preference matrix obtained by our algorithm at time $t$. Here, we additionally define $f_t^{'}(\Lambda_t, P^{(t)})$ to be the dual objective function of the online non-stationary problem:
$$
f_t^{'}(\Lambda_t, P^{(t)})= \mu \sum_{j=1}^m \overline{P_j}^{(t)} \log(Z_j(\Lambda_t, P^{(t)})) \phi_j(t) + \frac{1}{T} \langle \Lambda_t,b \rangle,
$$
where $\overline{P_j}^{(t)} = \max_{i}P^{(t)}_{ij}$, and $\phi_j(t)$ represents the ground truth probability that the next customer arrival is of type $j$ at time $t$.
The regret achieved at time $t$ is thus:
\begin{equation} 
\label{eqn:triangle_non_stationary}
\begin{aligned}
\mathrm{regret}^t_T &= f_t^{'} (\Lambda_t, P^{(t)})-f_t(\Lambda^*, P^\ast) \\ &\leq \vert f_t^{'}(\Lambda_t, P^{(t)})-f_t(\Lambda_t, P^{(t)})\vert + \vert f_t(\Lambda_t, P^{(t)})-f_t(\Lambda_t, P^\ast) \vert  \\ & \hspace{4mm} + \vert f_t(\Lambda_t, P^\ast)-f_t(\Lambda^*, P^\ast) \vert,
\end{aligned}
\end{equation}
where $\Lambda^*=\mathrm{argmin}_{\Lambda\in \kappa}{\sum_{t=1}^{T} f_t(\Lambda, P^\ast)}$ is the optimal solution in the offline problem of this time segment, and $\Lambda_t$ denotes our dual variable at time $t$. As before, $P^{(t)}$ is the preference matrix we have at time $t$ while $P^\ast$ is the ground truth preference matrix.

Using the same techniques as in the proof of \Cref{theorem:stationary_regret}, we can show that the average of the last two terms in \eqref{eqn:triangle_non_stationary} will converge to zero as $T \rightarrow \infty$. Hence, it suffices to show the convergence of the average of the first term, which captures the regret from approximating the non-stationary problem with a stationary problem. Recall that the dual objective of the online-stationary problem is defined as follows in \eqref{eqn:online_lp_3}:
$$
f_t(\Lambda_t, P^{(t)})= \mu \sum_{j=1}^m\frac{\lambda_j}{\sum_{s=1}^m \lambda_s} \overline{P_j}^{(t)} \log(Z_j(\Lambda_t, P^{(t)})) + \frac{1}{T} \langle \Lambda_t,b \rangle,
$$
where $$Z_j(\Lambda_t, P^{(t)})=\sum_{i=1}^n \exp \left( \frac{(r_i-\Lambda_{t,i})P_{ij}^{(t)}}{\mu \overline{P_j}^{(t)}} \right) 
\leq n \exp \left( \frac{r^\ast}{\mu}\right),$$
since $P^{(t)}_{ij} \leq \overline{P_{j}}^{(t)}$.
Hence, we must have $$\log(Z_j) \leq \log(n) + \frac{r^*}{\mu}\hspace{5mm}\forall j \in [m].$$
Additionally, note that 
$$\left\vert \frac{\lambda_j}{\sum_{s=1}^m \lambda_s}-\phi_j(t) \right\vert \leq \max{\left(\frac{\epsilon}{m \min_s{\lambda_s}}, \delta\right)},$$
where the first term on the right hand side results from time segments of type A, and the second term results from the confidence bound $\delta$ in time segments of type B. Realistically, the number of customers who arrive at an online marketplace per second are of the order of thousands or millions. Therefore, the customer arrival rates $\lambda_s$'s are of substantial magnitude. We thus assume that $\frac{\epsilon}{m \min_s{\lambda_s}} \ll \delta$. Now, if we define $R = m(\mu\log n + r^*),$ 
we can bound the difference between $f_t(\Lambda_t, P^{(t)})$ and $f_t^{'}(\Lambda_t, P^{(t)})$ with the following:
$$
\vert f_t^{'}(\Lambda_t, P^{(t)}) - f_t(\Lambda_t, P^{(t)}) \vert \leq R\delta,
$$
It follows that:
\begin{equation}
    \limsup\limits_{T \rightarrow \infty} \frac{\mathrm{regret}_T}{T} \leq \limsup\limits_{T \rightarrow \infty}\frac{\sum_{t=1}^T \vert f_t^{'}(\Lambda_t, P^{(t)}) - f_t(\Lambda_t, P^{(t)})\vert}{T} \leq R\delta.
\end{equation}
\end{proof}

From \Cref{theorem:non_stationary_regret}, we can see that the average regret does not converge to 0, but instead converges to a constant. In particular, $\delta$ corresponds to the confidence bound that we use in the type B time segments. Theoretically, by setting $\delta$ sufficiently small, we can control the average regret to converge to a number close to 0. Another trade-off certainly needs to be taken into account: as we decrease the regret, computational complexity will increase. However, as our numerical experiments later demonstrates, as long as we keep the value of confidence bound reasonably small, the average regret would tend to become negligible as the number of customers get larger. 

\section{Numerical Study}\label{sec:Experiment}
To test the efficacy of our proposed algorithms, we create different synthetic datasets that simulate customer preferences and arrivals following stationary or non-stationary Poisson processes. In order to make our linear programming problem non-trivial, we choose set-ups which guarantee that certain products will be sold out, while other products will have remaining budget in the optimal offline solution. In this section, we first apply \Cref{alg:integratedalg} to solve a set of online stationary problems and compare its empirical performance against the greedy heuristic. We then present experiments with non-stationary Poisson customer arrival processes and apply \Cref{alg:time} along with the integrated algorithm. The results of these experiments confirm the efficacy of our approach in tackling resource allocation problems with heterogeneous customer arrivals.

\subsection{Online Stationary Experiments} 
We first apply \Cref{alg:integratedalg} to a series of online stationary problems with varying numbers of customer arrivals. We compare its performance with the greedy heuristic, which simply matches each incoming customer to the product with largest reward available. The metric for evaluating algorithm performance is the average regret, as defined in \Cref{reg}.

We test our online stationary problem with four different sizes of total customer arrivals: $10^3, 10^4, 10^5, 10^6$. Each of them has the same initial set-up:
\begin{itemize}
    \item There are $m = 10$ types of customers and $d = 10$ products to be assigned. 
    \item The $j^{\mathrm{th}}$ type of customer arrives as a stationary Poisson process with constant arrival rate $\lambda_j = 0.1j$. 
    \item We draw the ground truth preference $P^\ast$ from a $\beta$ distribution, so that the buying behavior of different types of customers differ from each other. Our algorithm does not have any previous knowledge of customer preferences; instead, it learns preferences as the customers arrive.
    \item The budget of products is between 10\% and 30\% of the total population, in descending order. The rewards range from 0.1 to 1, in ascending order. Therefore, the products with higher rewards tend to have lower budgets. 
\end{itemize}
With the above set-up, we demonstrate the potential of integrated algorithm in solving a challenging resource allocation problem. 

\begin{figure}[htbp]
\centering
\includegraphics[width=0.9\textwidth]{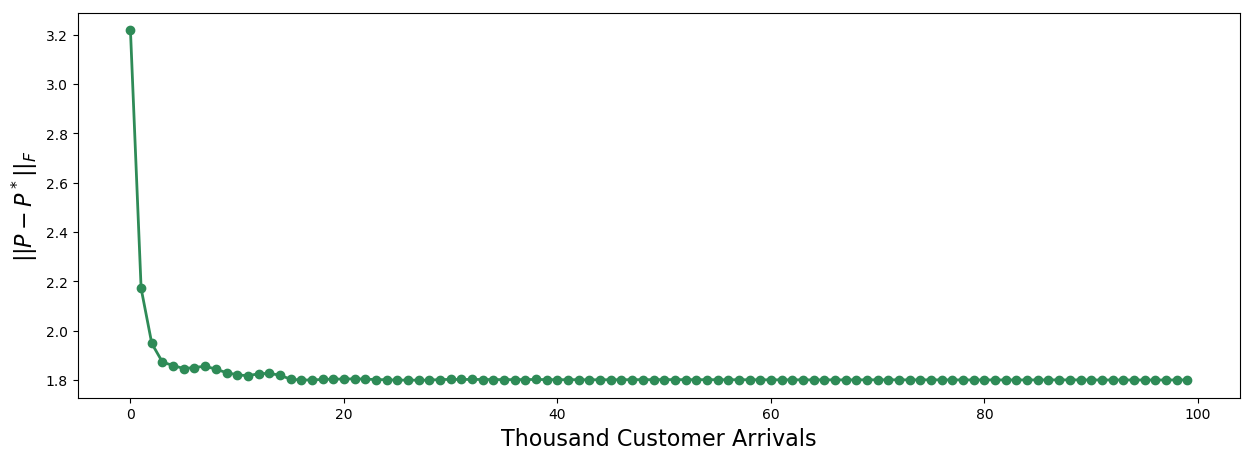}
\caption[Difference between $P^*$ and $P$]{The Frobenius norm between the ground truth $P^\ast$ and our preference matrix $P$ with every thousand customer arrivals.}
\label{fig:p_diff}
\end{figure} 

We first examine the algorithm's ability in learning customer preferences. Recall that in the beginning of the algorithm, we do not have any past knowledge of $P_{ij}$. As each customer arrives, we assign them to a product either by UCB or by the solution we reach from online GD. After a type $j$ customer gets assigned a product $i$, they will accept or decline the item based on their buying preference. Based on this new outcome, we update the entry $P_{ij}$ to more accurately reflect this customer's preference. We expect that as more customers arrive and get assigned to different types of products, our preference matrix $P$ will eventually converge to the ground truth matrix $P^\ast$. In the experiment with 100,000 customers, we perform the UCB algorithm for the first 20,000 incoming customers and rely on the solution from online GD afterwards. \Cref{fig:p_diff} shows the convergence of $P$ under this setting, in which we can see that $P$ hits the convergence horizon, approaching $P^\ast$ in the first 5,000 customer arrivals. After we stop applying the UCB algorithm, the value of $||P-P^\ast||_F$ remains stable because the customer preference that we have learned closely matches the actual purchasing behavior. We additionally note that $P$ no longer approaches $P^\ast$ quickly after the first few thousand arrivals. This is because after the application of the UCB algorithm in the beginning, the algorithm develops a good understanding of which customers have higher probabilities of purchasing certain products, thus avoiding matching those customers to products they are unlikely to buy. Therefore, it is difficult for the customer preference for those products to approach extreme accuracy. However, since most entries of $P$ and $P^\ast$ are sufficiently close, the remaining inaccuracy will not prevent the algorithm from making the optimal product allocation, and the regret introduced is minimal.

\begin{table}[ht]
\begin{center}
\renewcommand{\arraystretch}{1.6}
\setlength\tabcolsep{5pt}
\begin{tabular}{|| c || c | c | c | c ||}
 \hline
 \multirow{2}{*}{\makecell{Number of\\Customers}} & 
 \multirow{2}{*}{\makecell{Offline Dual Optimal\\Objective Value}} & \multirow{2}{*}{\makecell{Online Dual\\Objective Value}} & 
 \multirow{2}{*}{\makecell{Average\\Regret}} & 
 \multirow{2}{*}{Runtime} 
 \\ &&&& \\ \hline
 1,000 & 324.93 & 861.18 &  0.536 & 0.7s \\ \hline
 10,000 & 3251.91 & 4671.50 &  0.142 & 7.5s \\ \hline
 100,000 & 32480.30 & 29947.76 & 0.051 & 120s \\ \hline
 1,000,000 & 325171.14 & 295788.51 & 0.032 & 300s \\\hline
\end{tabular}
\caption{Results of online stationary experiments. We obtain the optimal offline dual objective by applying the first-order method $\mathrm{GD}^\mathrm{lin}$ \cite{allen2018lingering}.}
\label{tbl:experiment_stationary}
\end{center}
\end{table}

In \Cref{tbl:experiment_stationary}, we record the average regret and runtime obtained by \Cref{alg:integratedalg}. We can clearly see that the average regret decreases as the size of data gets larger. While the performance of the integrated algorithm is far from optimal in the first 1,000 customer arrivals, the dual variable already converges to the near-optimal solution when the size of the customer reaches 100,000, thus leading to a much smaller average regret. In each experiment, we choose to apply UCB and online GD enough times such that the regret will no longer exhibit drastic drops. We observe that in the experiment with 100,000 customer arrivals, after applying 20,000 rounds of UCB and 40,000 rounds of online GD, the dual variable already converges and requires no further computation. Therefore, we expect the runtime of the algorithm to remain at a considerably small value, as shown in the last column. 

\begin{figure*}[htbp]
\centering
\subfloat[Item selection by the greedy algorithm.]{%
  \includegraphics[width=0.45\textwidth]{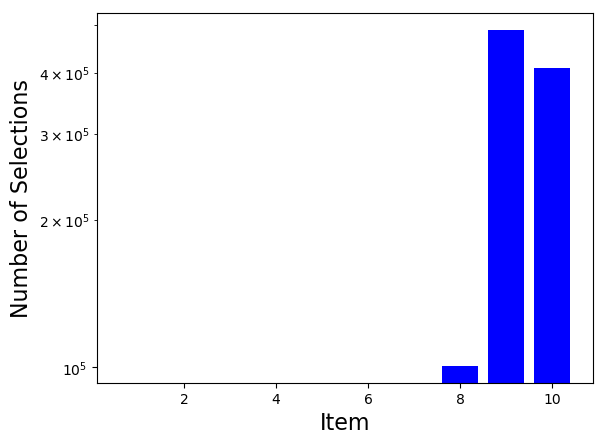}%
}\qquad
\subfloat[Item selection by \Cref{alg:integratedalg}.]{%
  \includegraphics[width=0.45\textwidth]{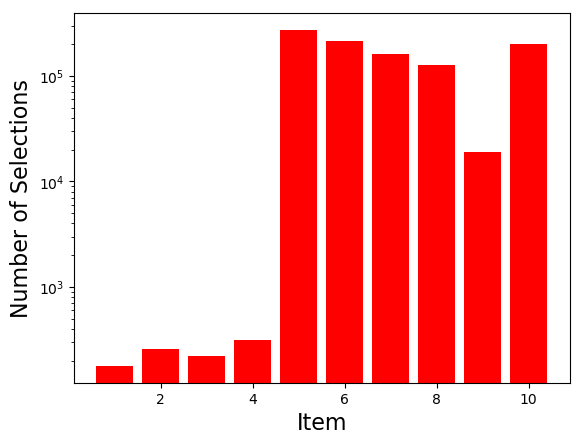}%
}
\caption{Number of times each product is selected.}
\label{fig:item_selection_comparison}
\end{figure*}

We have compared the results of \Cref{alg:integratedalg} with those of the greedy heuristic by directly comparing the revenue generated from the two approaches. In the greedy heuristic, each incoming customer is shown the product available with the highest reward until that product is fully consumed. In \Cref{fig:item_selection_comparison}, we show the number of times that each product is presented to customers by the greedy algorithm and the integrated algorithm, respectively. We can tell that the integrated algorithm is not greedy since it does not select an item solely based on its reward value. However, each customer might have a different preference for the product with highest reward, so intuitively we expect this approach to be somewhat naive and not necessarily to lead to the optimal outcome. 

\begin{table}[ht]
\begin{center}
\renewcommand{\arraystretch}{1.6}
\setlength\tabcolsep{5pt}
\begin{tabular}{|| c || c | c | c ||}
 \hline
 \multirow{2}{*}{\makecell{Number of \\ Customers}} & 
 \multirow{2}{*}{\makecell{Offline\\Revenue}} & \multirow{2}{*}{\makecell{Greedy Algorithm\\Revenue}} & 
 \multirow{2}{*}{\makecell{Integrated Algorithm\\Revenue}}
 \\ &&& \\ \hline
 1,000 & 324.93 & 198.00 & 154.00 \\  \hline
 10,000 & 3251.81 & 2024.80 & 2841.40 \\  \hline
 100,000 & 32413.36 & 20365.60 & 31438.10 \\  \hline
 1,000,000 & 324240.92 & 203584.80 & 315435.10 \\ \hline
\end{tabular}
\caption{Revenues generated with the offline approach, the greedy algorithm and \Cref{alg:integratedalg}.}
\label{tbl:compare_with_greedy}
\end{center}
\end{table}

This disparity in customer preferences is also indicated by our experiment results, which are shown in \Cref{tbl:compare_with_greedy}. The first column records the optimal revenue that we can achieve if we are to solve the corresponding offline problem. When the number of customer arrivals is small (e.g., 1,000), the greedy approach gives a higher revenue than the integrated algorithm; this is because the integrated algorithm has not gone through a sufficient number of online GD iterations for the dual variable to converge. As the sizes of data later increases, we observe that the revenue generated by the integrated algorithm is closer to the optimal revenue achieved in the offline problem and also exceeds that of the greedy heuristics. We additionally note that there can be cases where the greedy heuristics might give better performance. For instance, when the customer preference for each product are close to each other, choosing the product with the highest reward is essentially the optimal solution. However, since the customer preference in realistic settings tend to have more variance, the integrated algorithm would almost always allocate the better product.

\subsection{Online Non-Stationary Experiments}

We now move on to test the performance of the proposed online non-stationary algorithm, which combines \Cref{alg:integratedalg} and \Cref{alg:time}. Recall that we do this by converting the non-stationary problem into a series of stationary problems and then solving each of the stationary problems accordingly. In this subsection, we present two representative experiments, each having initial set-ups that make the problem non-trivial: the first experiment comes with extreme budget constraints, while the second is closer to a realistic setting, where each product comes with diverse reward.

\subsubsection{Experiments with Extreme Budget Constraints}
In the first set of experiments for the non-stationary problem, the following set up is considered:
\begin{itemize}
\item There are $m = 10$ types of customers and $d = 10$ products to be assigned. 
\item Each type of customer is associated with an arrival rate function that changes fairly slowly. \Cref{fig:rate_functions} shows some example rate functions that we consider.
\item We draw the ground truth customer preference matrix $P^\ast$ randomly from a Gaussian distribution centered around 0.1. 
\item One product has infinite budget, two have small budgets (10\% of the population) and the rest have minimal budgets (1\% of the population).
\item The reward of each product is set to be uniformly 1.
\end{itemize}

\begin{figure}[htbp]
\centering
\includegraphics[width=0.9\textwidth]{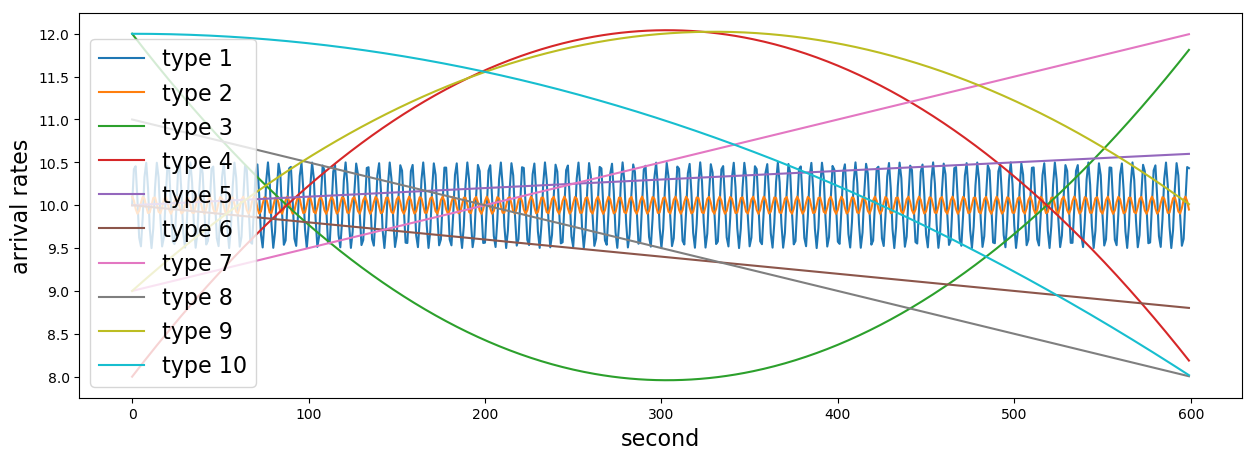}
\caption[Rate Functions used in the experiment with 60,000 customer arrivals]{The rate functions $\lambda_j(t)$ of each type of customer used in the experiment with 60,000 customer arrivals. There are two trig functions, four linear functions and four quadratic functions.}
\label{fig:rate_functions}
\end{figure} 

\begin{figure*}[htbp]
\centering
\subfloat[Customers Type 4.]{%
  \includegraphics[width=0.45\textwidth]{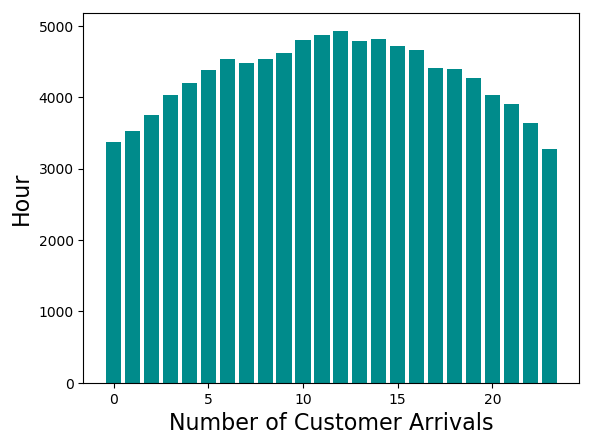}%
}\qquad
\subfloat[Customer Type 9.]{%
  \includegraphics[width=0.45\textwidth]{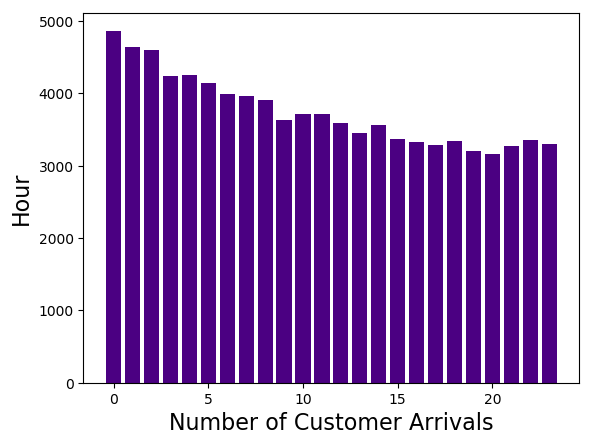}%
}
\caption{Number of customer arrivals each hour in the 1M experiment.}
\label{fig:arrivals}
\end{figure*}

We have performed experiments using three different sizes of data, which includes: (1) a population of 6,000 people arriving in an hour, (2) a population of 60,000 people arriving in 10 hours, and (3) a population of 1,000,000 people arriving in 24 hours. We have scaled the rate functions in accordance with the length of the time span to keep the experimental set-ups consistent. In \Cref{fig:arrivals}, we plot the number of customer arrivals each hour in the experiment with 1,000,000 customers for two different types of customers. We can clearly observe that the numbers vary with the hours, and meanwhile, different types of customers have different arrival patterns. 

In \Cref{fig:items_extreme}, we plot the number of each product assigned to the customers in the experiment with 1,000,000 customer arrivals. The result is as expected: all the products with minimal or large budgets have been sold to customers who have higher preference for those products, and the only product that has remaining budget is the one with infinite budget. As before, we not only care about the assignment of items, but also how close our dual variable is to the optimal solution. We thus move on to compute the online dual objective and evaluate its performance via average regret.

\begin{figure}[ht]
\centering
\includegraphics[width=0.9\textwidth]{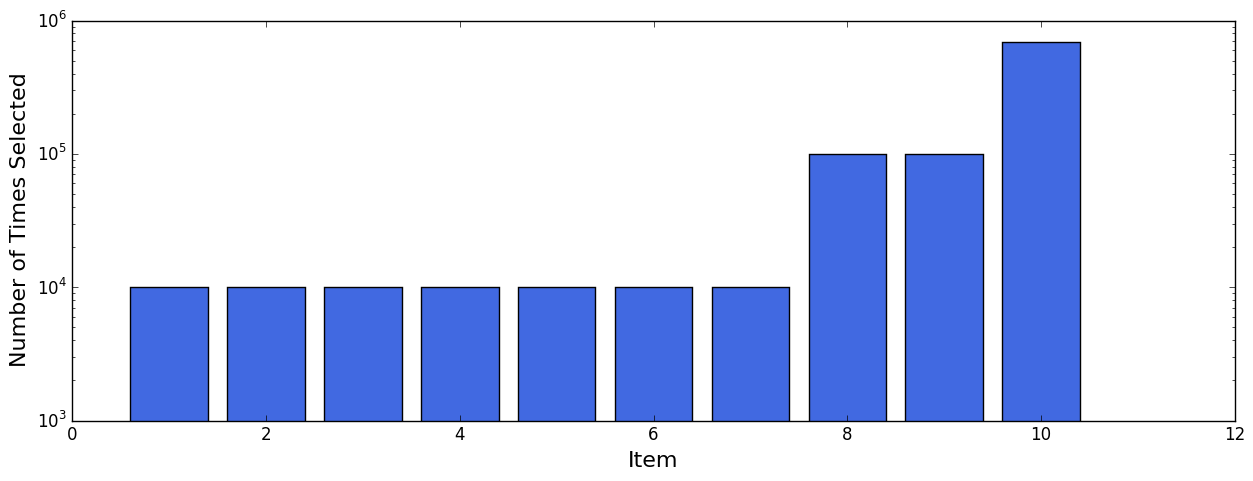}
\caption[Assignment with extreme budget constraints]{The item assignment of the 1M experiment with extreme budget constraints. Items are sorted in ascending order by values of their budget constraints.}
\label{fig:items_extreme}
\end{figure} 

\Cref{tbl:experiment1} records the results of the experiments across all three different sizes of data. We have computed the optimal offline dual objective mentioned in \Cref{subsec:offline}. We computed the online dual objective using \eqref{eq:solvefun}, and the total regret using \Cref{reg}. We observe the decrease of average regret as we increase the size of data, which confirms the theoretical result in \cref{theorem:non_stationary_regret}. When applying the non-stationary algorithm, we set the number of rounds of UCB and online GD we wish to apply inside each time segment, which mainly determines the runtime of the algorithm. Here, as we make the size of the data larger, we increase the number of rounds of UCB and gradient computations accordingly. The runtime hence scales up roughly linearly. 

\begin{table}[htbp!]
\begin{center}
\renewcommand{\arraystretch}{1.6}
\setlength\tabcolsep{5pt}
\begin{tabular}{|| c || c | c | c | c ||}
 \hline
 \multirow{2}{*}{\makecell{Number of\\Customers}} & 
 \multirow{2}{*}{\makecell{Offline Dual Optimal\\Objective Value}} & \multirow{2}{*}{\makecell{Online Dual\\Objective Value}} & 
 \multirow{2}{*}{\makecell{Average\\Regret}} & 
 \multirow{2}{*}{Runtime} 
 \\ &&&& \\ \hline
 6,000 & 612.54 & 900.58 & 0.0454 & 5s \\ \hline
 60,000 & 6222.43 & 6843.15 & 0.0101 & 40s \\ \hline
 1,000,000 & 98547.03 & 105854.78 & 0.0076 & 940s \\ \hline
\end{tabular}
\caption[Extreme budget constraints non-stationary experiment]{Results of experiment with extreme budget constraints. We obtain the optimal offline dual objective by applying $\mathrm{GD}^\mathrm{lin}$. }
\label{tbl:experiment1}
\end{center}
\end{table}

\subsubsection{Experiments with Varying Rewards}
We have also performed another experiment with varying reward values such that the assignment of the items are not solely based on their budget constraints and customer preference, which reflects a more realistic setting. The set-ups of this experiment remain the same as the previous experiment, with the following exceptions:
\begin{itemize}
\item Instead of applying the extreme budgets constraints as before, we select one product to have fairly large budget (66.67\% of the population) and let the rest of the products have fairly small budgets (10\% of the population).
\item The rewards of products vary from 0.2 to 1. In particular,the product with the most budgets is associated with a reward value of 0.2, so this product should be the least favorable one to most customers. 
\end{itemize}

\begin{figure}[htbp]
\centering
\includegraphics[width=0.9\textwidth]{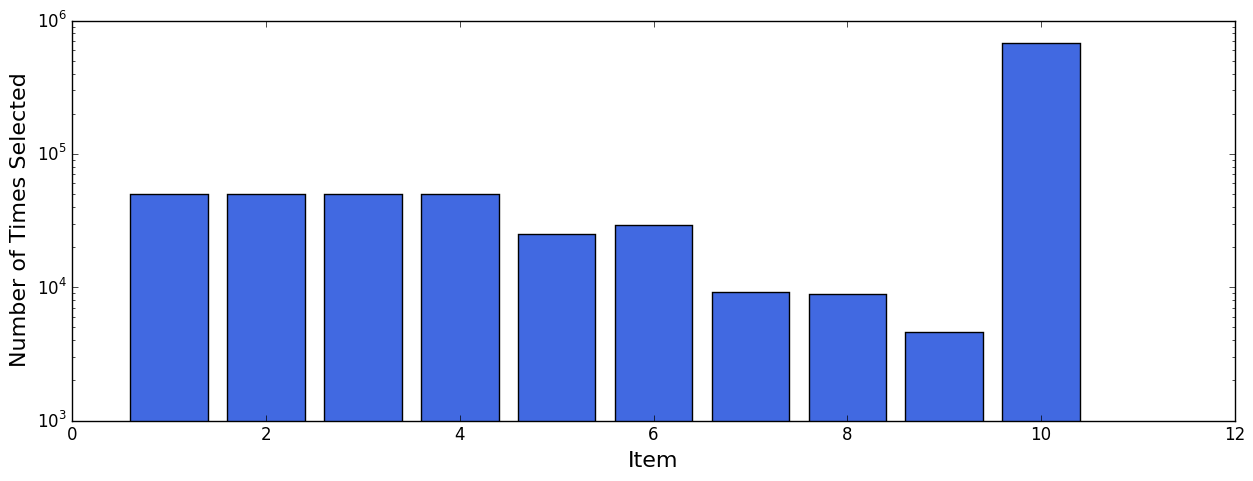}
\caption[Assignment with varying rewards]{The item assignment of the experiment with 1,000,000 customer arrivals and varying rewards. Items are sorted in descending order by values of their rewards.}
\label{fig:items_varying}
\end{figure}  

The result of the experiment with 1,000,000 customer arrivals can be seen in \Cref{fig:items_varying}. Observe that our algorithm ensures that the four items with highest rewards have been fully sold out, while the rest still have remaining budgets in the end. It is noteworthy that even the items with low rewards have been sold to some extent, due to the exploration component of the UCB algorithm. Overall, such a solution matches our expectation of a near-optimal solution. 

We have again compared the results across three different sizes of data, which can be seen in \Cref{tbl:experiment2}. As before, we can observe the average regret decreasing as the number of customer arrivals increase. The runtime is similar to the first set of experiments, which are reasonable in terms of the size of data. Overall, our experiments have shown that the non-stationary algorithm has a good potential of being applied to real-world online product allocation problems. 

\begin{table}[htbp]
\begin{center}
\renewcommand{\arraystretch}{1.6}
\setlength\tabcolsep{5pt}
\begin{tabular}{|| c || c | c | c | c ||}
 \hline
 \multirow{2}{*}{\makecell{Number of\\Customers}} & 
 \multirow{2}{*}{\makecell{Offline Dual Optimal\\Objective Value}} & \multirow{2}{*}{\makecell{Online Dual\\Objective Value}} & 
 \multirow{2}{*}{\makecell{Average\\Regret}} & 
 \multirow{2}{*}{Runtime} 
 \\ &&&& \\ \hline
 6,000 &  608.97  & 549.35 & 0.0199 & 4s \\ \hline
 60,000 & 6217.12  & 5539.41 &  0.0114 & 36s \\ \hline
 1,000,000 & 97060.83 & 95482.46 & 0.0025 & 957s \\ \hline
\end{tabular}
\caption[Non-stationary experiment with varying rewards]{Results of experiment with varying rewards. }
\label{tbl:experiment2}
\end{center}
\end{table}

\section{Conclusion}
\label{sec:Conclusion}
In this work, we propose algorithms that tackle the online resource allocation problem, in which we aim to recommend each customer with an item in ways that not only maximize potential reward, but also satisfy budget constraints. In order to find the optimal solution to our online objective function, we first must learn the preferences, or $P_{ij}$, for each customer type. To learn the probability that a customer in a certain type purchases a given item, we use the Upper Confidence Bound (UCB) algorithm, which decides which item to recommend to a customer. When the customer arrives, we observe whether or not they have purchased the item recommended to them and update our customer preference variable, $P_{ij}$. We incorporate this value into our objective function, and apply online Gradient Descent to minimize our dual function. Over time, as more customers arrive, the estimations for the $P_{ij}$ values become more accurate, and the UCB algorithm is able to make better recommendations, ones that have a higher probability of reward. Overall, our online stationary algorithm combines reinforcement learning with online optimization to minimize our dual function and find the optimal solution. Our tests on this novel algorithm have supported our theory that regret of this algorithm approaches zero when the number of customers is sufficiently large, and that our algorithm produces a better solution than greedy heuristics. 

Although our online stationary algorithm performs well, customers do not always arrive following a stationary Poisson process. In a more realistic scenario, the rate at which customers arrive varies over time. This motivates us to consider the online non-stationary problem, we consider customers arriving onto the webpage following a non-stationary Poisson process. When we remove the assumption of customer arrival following a stationary Poisson processes, we are met with complexity in formulating the online LP, as this type of problem cannot simply be solved using the proposed online stationary algorithm. An additional difficulty lies in that if we do online learning for $P_{ij}$, the dual variables will not necessarily converge. Moreover, if the arrival processes are modeled as non-stationary Poisson processes, then the probability that the new arriving customer comes from type $j$ is almost impossible to compute. Our non-stationary algorithm approaches these difficulties by dividing the non-stationary problem into several stationary problem, under the assumption that the arrival rate functions of different customers change fairly slowly in a small time segment. We have shown theoretically that the regret of the online non-stationary algorithm should approach a small value near zero when the number of customers is sufficiently large. Our empirical results with both extreme budget constraints and non-trivial budget constraints also support this theoretical result.

\section{Future Work}
\label{sec:FutureWork}

There are many rich, exciting directions that one can pursue with this work. In the product recommendation model we propose above, we have considered a rather simplified scenario, matching each customer with one product at a time and aiming to maximize the expected profit brought by this assignment. However, one can in fact make the current model more realistic by introducing more complications: (1) When each customer arrives, an e-commerce platform can in fact display a set of products to the customer at the same time. (2) The user engagement that a website wishes to maximize is not necessarily the reward values, but the number of clicks or the dwell time that a user spend on the webpages. (3) Sometimes there are more business contraints to consider, e.g., one needs to guarantee a fixed number of selections for a certain product. A potential future direction of this work is to take the additional settings above into the model construction, and develop variants of the proposed algorithms that can deal with these more complicated situations.

In addition, improvements can also be made towards the performance of the proposed online stationary and non-stationary algorithms. While the runtime of the algorithms are reasonable considering the large scale of the data, one might further decrease the runtime of these algorithms with the application of parallel computing. This would enable our algorithms to have the potential of being applied in real-world settings, where e-commerce companies oftentimes need to deal with even larger scale of customer arrivals in a shorter period of time, e.g., a million customer arrivals within a second. Throughout our analysis of algorithm performance, we have only tested our algorithms with synthetic data; therefore, we are also interested in understanding how they perform when dealing with real-world datasets.

\appendix
\section{Notations} \label{app:Notations}
\Cref{tbl:TableOfNotation} records symbols used throughout the paper.

\begin{table}[htbp]
\begin{center}% used the environment to augment the vertical space
% between the caption and the table
\begin{tabular}{r |c p{9cm}}
$n$ &  & Number of items \\
$m$ &  & Number of customer types \\
$i$ &  & Item indices, $i \in [1,n]$\\
$j$ &  & Customer type indices , $j \in [1,m]$\\
$r_i$ &  & Reward in terms of revenue for the company for a given customer buying certain item $i$ \\
$b_i$ && Budget constraint of item $i$, $b \in \mathbb{R}^{n}$\\
$x_{ij}$ &  & Probability that a customer of type $j$ gets recommended item $i$ \\  
$P_{ij}$ & & Probability that a customer of type $j$ will buy item $i$ given that they are offered item $i$\\
$\overline{P_j}$ && For a given customer of type $j$, his highest possibility of buying any particular product, i.e., $\overline{P_j} = \max_{i}P_{ij}$\\
$\mu$ & & Regularization parameter \\
$\Lambda$ && Dual variable vector of dimension $n$ \\
$\eta_t$ && Step size in optimization algorithm at an iteration\\

\end{tabular}
\end{center}
\caption{Table of Notations}
\label{tbl:TableOfNotation} 
\end{table}

\section*{Acknowledgments}
This project was completed during the Research in Industrial Projects for Students (RIPS) 2019, under the sponsorship of the Institute for Pure and Applied Math (IPAM) at UCLA and the Alibaba Group. We would like to thank our academic and industry mentors Anna Ma, Xinshang Wang, and Wotao Yin for their help discussions. We would also like to thank Susana Serna, Dima Shlyakhtenko and all of the IPAM staff who made RIPS 2019 possible. 

\bibliographystyle{siamplain}
\bibliography{references}

\end{document}